\documentclass[11pt,reqno]{amsart}
\usepackage{indentfirst, amssymb, amsmath, amsthm, mathrsfs, setspace, indentfirst, enumerate,  mathrsfs, amsmath, amsthm}
\usepackage[bookmarksnumbered, colorlinks, plainpages]{hyperref}
\usepackage{mathrsfs}
\usepackage{cite}
\usepackage{tikz}
\usepackage{graphicx}
\usepackage{placeins}
\usetikzlibrary{arrows.meta,calc,decorations.pathreplacing,positioning,shadows.blur,backgrounds,calc,fit,
	decorations.pathreplacing}
\usetikzlibrary{shapes.geometric}
\usepackage{float}
\usepackage{pgfplots}
\usepackage{booktabs}
\usepackage{threeparttable} 
\usetikzlibrary{shadows}
\pgfplotsset{compat=1.18}
\usepackage[table]{xcolor}
\usepackage{array}
\usepackage{makecell}
\pgfplotsset{compat=1.18}
\usepackage{smartdiagram}
\newtheorem*{cor A}{Corollary A}
\newtheorem*{cor B}{Corollary B}

\newtheorem{theo}{Theorem}[section]
\newtheorem{lem}{Lemma}[section]
\newtheorem{cor}{Corollary}[section]

\newtheorem{defi}{Definition}[section]
\newtheorem{rem}{Remark}[section]
\newcommand{\ol}{\overline}
\newcommand{\be}{\begin{equation}}
	\newcommand{\ee}{\end{equation}}
\newcommand{\beas}{\begin{eqnarray*}}
	\newcommand{\eeas}{\end{eqnarray*}}
\newcommand{\bea}{\begin{eqnarray}}
	\newcommand{\eea}{\end{eqnarray}}

\numberwithin{equation}{section}

\begin{document}
\title[T\MakeLowercase {oeplitz determinants for starlike}......]{\LARGE T\Large\MakeLowercase {oeplitz determinants for starlike mappings of order $\alpha$ in several complex variables}}

\date{}
\author[S. P\MakeLowercase{anja}, S. M\MakeLowercase {ajumder} \MakeLowercase{and} A. B\MakeLowercase {anerjee}  ]{S\MakeLowercase{hantanu} P\MakeLowercase{anja}$^1$$^*$, S\MakeLowercase {ujoy} M\MakeLowercase {ajumder}$^2$ \MakeLowercase{and} A\MakeLowercase{bhijit} B\MakeLowercase {anerjee}$^3$}

\address{$^1$ Department of Mathematics, University of Kalyani, West Bengal 741235, India.}
\email{panjasantu07@gmail.com, shantanumath26@klyuniv.ac.in}

\address{$^{2}$ Department of Mathematics, Raiganj University, Raiganj, West Bengal-733134, India.}
\email{sm05math@gmail.com, sjm@raiganjuniversity.ac.in}

\address{$^{3}$ Department of Mathematics, University of Kalyani, West Bengal 741235, India.}
\email{abanerjee\_kal@yahoo.co.in, abhijitbanerjee@klyuniv.ac.in}

\renewcommand{\thefootnote}{}
\footnote{*\emph{Corresponding Author}: Shantanu Panja}

\renewcommand{\thefootnote}{\arabic{footnote}}
\setcounter{footnote}{0}

\begin{abstract}
	Let $F(z)=zf(z)$, where $f$ is a normalized scalar holomorphic function on the unit ball $\mathbb B$ of a complex Banach space $X$ or on the unit polydisk $\mathbb U^n$ of $\mathbb{C}^n$. We establish sharp bounds for two Toeplitz determinants involving the second, third and fourth Fr\'echet derivatives of $F$ when $F$ is starlike of order $\alpha$, $0\leq\alpha<1$. The first bound is sharp for the full range $0\leq\alpha<1$, while the second is sharp for $0\leq\alpha\leq 2/3$. The extremal mappings are explicitly identified, confirming the sharpness of the estimates.
\end{abstract}

\thanks{Typeset by \AmS -\LaTeX}
\maketitle

\section{\bf Introduction}

Toeplitz determinants are closely related to the coefficients of analytic and univalent functions. They provide nonlinear combinations of Taylor coefficients and therefore give a natural way to study the structure of
important subclasses of univalent functions. Such determinants have also appeared in several areas of mathematics; see, for example, the survey of
Ye--Lim \cite{Ye-Lim}. In geometric function theory, their
study is closely connected with sharp coefficient estimates for important subclasses of univalent functions. 

Let $\mathcal{S}(\mathbb{U})$ denote the class of analytic functions defined in the unit disk $\mathbb{U}=\{z\in\mathbb{C}:|z|<1\}$, normalized by the Taylor series expansion
\begin{align}\label{eq1}
	f(z)=z+\sum_{n=2}^{\infty}a_nz^n,\qquad z\in\mathbb{U}.
\end{align} 
\begin{defi} For $q\geq1$ and $n\geq 1$, the $q$th Toeplitz determinant $T_q(n)$ associated with $f$ is defined by
\begin{align*}
		T_q(n)(f)=
		\begin{vmatrix}
			a_{n} & a_{n+1} & \cdots & a_{n+q-1} \\
			a_{n+1} & a_{n} & \cdots & a_{n+q-2} \\
			\vdots & \vdots & \ddots & \vdots \\
			a_{n+q-1} & a_{n+q-2} & \cdots & a_{n}
		\end{vmatrix}.
	\end{align*}
	
\end{defi}

The two determinants considered in this paper are
\begin{align*}
T_2(3)(f)=a_3^2-a_4^2\quad\text{and}\quad T_3(2)(f)=(a_2-a_4)(a_2^2+a_2a_4-2a_3^2).
\end{align*}

Thus, the estimates of these determinants involve nonlinear
interactions among several consecutive Taylor coefficients.

For the class of starlike functions of order $\alpha$,
\begin{align*}
\mathcal S_{\alpha}^*(\mathbb{U})=
\left\lbrace f\in\mathcal{S}(\mathbb{U}):\Re\left\lbrace \frac{zf'(z)}{f(z)}\right\rbrace>\alpha
\right\rbrace,\qquad 0\leq\alpha<1.
\end{align*}
For $\alpha=0$, Ali et al. \cite{Ali et al.} established the sharp bounds
 \begin{align*}
 	|T_2(3)(f)|\le 25\quad\text{and}\quad|T_3(2)(f)|\le 84.
 \end{align*}
Ahuja et al.~\cite{Ahuja et al.} obtained the sharp estimates
\begin{align*}
	|T_2(2)(f)|
	\leq
	(1-\alpha)^2(13-12\alpha+4\alpha^2),
\end{align*}
 and
\begin{align*}
	|T_3(1)(f)|
	\leq
	12\alpha^4-52\alpha^3+91\alpha^2-74\alpha+24\quad\text{for}\quad 0\leq\alpha\leq\frac{2}{3}.
\end{align*} 
\par
We note that Ahuja et al. \cite{Ahuja et al.} estimates concern the different
functionals $T_2(2)(f)$ and $T_3(1)(f)$; sharp order-$\alpha$ bounds for $T_2(3)(f)$ and $T_3(2)(f)$ themselves do not appear to be available in the literature. These one-variable estimates provide a useful starting point for the higher-dimensional problem. However, the passage from one complex variable to several complex variables is not direct. In one variable, the expression $zf'(z)/f(z)$ gives a simple characterization of starlikeness. In a complex Banach space, the corresponding condition involves Fr\'echet derivatives and continuous linear functionals.
Moreover, the geometry of the domain and the choice of the norm enter the coefficient problem. Thus, the corresponding Toeplitz determinant problem requires a different treatment.

There has been considerable progress on coefficient estimates for normalized biholomorphic mappings in several complex variables. Important results in this direction were obtained by Bracci \cite{Bracci-2015}, Bracci and Roth
\cite{Bracci-2018}, Graham et al.~\cite{Graham-2002}, Hamada and Honda \cite{Hamada-Honda-2008}, Hamada et al.~\cite{Hamada-Honda-Kohr-2006} and Xu et al.~\cite{Xu-Liu-Liu-2018}. These works show that many questions from one-variable geometric function theory have meaningful higher-dimensional counterparts.

Toeplitz determinants have also recently been considered for holomorphic mappings in several variables. Giri--Kumar \cite{Giri-Kumar} obtained sharp bounds for $T_2(1)(f)$ and $T_3(1)(f)$ for suitable classes of holomorphic
mappings on the unit ball of a complex Banach space and on the unit polydisk. Xu et al.~\cite{Xu et al.} later considered $k$-fold symmetric holomorphic mappings. Other related estimates can be found in \cite{Giri, Xu+Jiang}.
\subsection*{Motivations}
The higher-dimensional behaviour of the second order Toeplitz
determinants $T_2(3)$ and $T_3(2)$ for normalized starlike mappings of order $\alpha$ has not been fully described in the settings considered here.
This leads to the main questions of the present paper: how do the sharp one-variable bounds extend to the unit ball of a complex Banach space and to the unit polydisk?

To make the higher-dimensional problem tractable while retaining sharpness, we work with the natural subclass of starlike mappings of order $\alpha$ obtained by radicalizing a scalar holomorphic function, namely $F(z)=zf(z)$ with $f(0)=1$; this construction is
standard in the study of extremal problems for starlike mappings in several variables (cf.\ \cite{Graham-Kohr}).

We answer these questions by obtaining sharp bounds for $T_2(3)$ and $T_3(2)$ for normalized starlike mappings of order $\alpha$ on the unit ball of a complex Banach space and on the unit polydisk in $\mathbb C^n$.
The results recover the corresponding one-variable estimates when the dimension is one. In particular, the results give higher-dimensional versions of the known one-variable estimates and show that the same coefficient combinations remain meaningful in the Banach space setting.

\section*{\bf Preliminaries}

Before proceeding further, let us establish some fundamental terminology that will recur throughout this work. Let $X$ denote a complex Banach space with norm $||\cdot||$. The ball of radius $r$ is denoted by $\mathbb{B}_r =\{z\in X:||z||<r\}$. When $r=1$, we will simply use $\mathbb{B}$ and if $X=\mathbb{C}$, then $\mathbb{B}=\mathbb{U}$. Let $\partial \mathbb{B}$ denote the boundary of $\mathbb{B}$, $\overline{\mathbb{B}}$ its closure, and $\partial_0 \mathbb{U}^n$ the characteristic boundary (i.e., the boundary on which the maximum modulus of a holomorphic function is attained) of the unit polydisk $\mathbb{U}^n$.

Now let $X$ and $Y$ be two complex Banach spaces with respect to the norms $||.||_1$ and $||.||_2$ respectively. For simplicity, we denote both norms by $||.||$, when there is no possibility of confusion. Let $L(X, Y)$ denote the set of continuous linear operators from a complex Banach space $X$ into a complex Banach space $Y$ with the standard operator norm 
\begin{align*} 
||A||=\sup\{||A(z)||: ||z||=1,\quad A\in L(X,Y)\}.
\end{align*} 

Let $I$ be the identity operator in $L(X, X)$. For each $z \in X \setminus \{0\}$, define the set  
\begin{align*}
	\displaystyle 
	T(z) = \{ l_z \in L(X, \mathbb{C}) : l_z(z) = \Vert{}z\Vert{}, \quad \Vert{}l_z\Vert{} = 1 \}.
\end{align*}

This set is guaranteed to be nonempty by the Hahn-Banach theorem. 
Let $\Omega$ be a domain in $X$ and let $f:\Omega \rightarrow Y$ be a mapping. Then $f$ is called holomorphic on $\Omega$ if for any $z\in\Omega$ there is a mapping $Df(z)\in L(X,Y)$, called the Fr\'echet derivative of $f$ at $z$ such that 
\begin{align*} \lim_{h\to 0} \frac{||f(z+h)-f(z)-Df(z)h||}{||h||}=0,\end{align*} i.e., 
\begin{align*} 
f(z+h)=f(z)+Df(z)h+o(||h||).
\end{align*}

Let $\mathcal{H}(\Omega, \Omega')$ represent the set of all holomorphic mappings from a domain $\Omega \subset X$ into another domain $\Omega' \subset Y$. We simply write $\mathcal{H}(\Omega, X) = \mathcal{H}(\Omega)$. It is well known that if $f \in \mathcal{H}(\Omega,Y)$, then for all $w$ in some neighborhood of $z \in \Omega$, the function can be expressed as the power series expansion:
\begin{align*}
\displaystyle
f(w) = \sum_{n=0}^{\infty} \frac{1}{n!} D^n f(z)((w - z)^n),
\end{align*}
where $D^n f(z)$ is a bounded symmetric $n$-linear mapping from $\prod_{j=1}^n X$ into $Y$, called the $n$-th Fr\'echet derivative of $f$ at $z$. Moreover, the term $D^n f(z)((w - z)^n)$ is shorthand for applying the $n$-linear mapping to $n$ copies of the vector $(w - z)$:
\begin{align*}
	\displaystyle
	D^n f(z)((w - z)^n) = D^n f(z)(\underbrace{w - z, w - z, \dots, w - z}_{n \text{ times}})
\end{align*}
where $D^0f(z)=f(z)$.\vspace{1.2mm}

A mapping $f \in \mathcal{H}(\mathbb{B})$ is said to be normalized if $f(0) = 0$ and $Df(0) = I$, where $I \in L(X,X)$ is the identity operator. It is called biholomorphic if $f$ is bijective with a holomorphic inverse $f^{-1}$ on $f(\mathbb{B})$ and locally biholomorphic if $Df(z)$ has a bounded inverse for every $z \in \mathbb{B}$. The family of all normalized biholomorphic mappings on $\mathbb{B}$ is denoted by $\mathcal{S}(\mathbb{B})$. \vspace{1.2mm}

A case of particular interest is that of the $p$-norm in
$\mathbb{C}^{n}$, that is,
\begin{align*}
	\displaystyle
	\|z\|_{p}=
	\begin{cases}
		\left(
		\displaystyle\sum_{j=1}^{n}|z_{j}|^{p}
		\right)^{1/p},
		& 1\le p<\infty, \\[2ex]
		\displaystyle\max_{1\le j\le n}|z_{j}|,
		& p=\infty.
	\end{cases}
\end{align*}

Let $\mathbb{B}(p)$ denote the unit ball with respect to $\|\cdot\|_{p}$.
When $p=\infty$, $\mathbb{B}(p)$ reduces to the unit polydisk $\mathbb{U}^n$ of $\mathbb{C}^{n}$. It is known that if $X=\mathbb{C}^{n}$ with respect to the maximum norm $\Vert{}.\Vert{}_\infty $, and $l_z\in T(z)$, then
\begin{align}\label{eq3}
	\displaystyle
	l_z(w)=
	\sum_{\{j:\|z\|_{\infty}=|z_j|\}}
	t_j\,
	\frac{w_j\overline{z_j}}
	{\|z\|_{\infty}},
\end{align}
where each $t_j\geq0$ and $\sum\limits_{\{j:\|z\|_{\infty}=|z_j|\}} t_j=1.$\vspace{1.2mm}

The following family plays a key role in our discussion:
\begin{align*}
	\displaystyle
	\mathcal{M}=\{f\in\mathcal{S}(\mathbb{B}),\; \Re\{l_z(f(z))\}>0,\; z\in\mathbb{B}\backslash\{0\},\; l_z\in T(z)\}.
\end{align*}

In the case of the maximum norm $\Vert{}.\Vert{}_\infty $ on $\mathbb{C}^n$, the class $\mathcal{M}$ consists of all normalized holomorphic mappings $w:\mathbb{U}^n \to \mathbb{C}^n$ satisfying
\begin{align*}
	\displaystyle
	\Re\left\lbrace \frac{w_j(z)}{z_j} \right\rbrace >0,
\end{align*}
whenever $\Vert{}z\Vert{}_\infty = \vert{}z_j\vert{} > 0$ for $1 \le j \le n$.
 
\section{{\bf Toeplitz determinants for starlike mappings of order $\alpha$ on a complex Banach space}}\label{Sec-2}
In this section, we recall the definition of starlikeness of order $\alpha$ in a complex Banach space.
\begin{defi}\emph{\cite[Definition 1.2]{Hamada-2001}}
	Let $f:\mathbb{B}\to X$ be a normalized locally biholomorphic mapping and let $\alpha\in [0,1)$. We say that $f$ is starlike of order $\alpha$ if 
	\begin{align}\label{SL1}
		\displaystyle
		\Re\left\lbrace \frac{||z||}{l_z([Df(z)]^{-1}f(z))}\right\rbrace>\alpha,\quad z\in\mathbb{B}\backslash\{0\},\quad l_z\in T(z).
	\end{align}
\end{defi}

Let $\mathcal{S}^*_{\alpha}(\mathbb{B})$ denote the class of starlike mappings of order $\alpha$ on $\mathbb{B}$. When $\alpha=0$, the class $\mathcal{S}^*_{\alpha}(\mathbb{B})$ reduces to the standard starlike class $\mathcal{S}^*(\mathbb{B})$.

 The theory of starlike mappings in higher dimensions has been extensively studied. For more information, see the recent monographs by Gong  \cite{Gong} and Kohr--Liczberski  \cite{Kohr- Liczberski-1998}.
 
 By considering biholomorphic mappings $F:\mathbb{B}\to X$ of the form $F(z)=f(z)z$ ($z\in\mathbb{B}$), where $f:\mathbb{B}\to \mathbb{C}$ is a holomorphic function with $f(0)=1$, we obtain the following results.
 
\begin{theo}\label{Th-1} Let $f\in\mathcal{H}(\mathbb{B},\mathbb{C})$ satisfy $f(0)=1$, and suppose that the mapping $F:\mathbb{B}\to X$ defined by 	$F(z)=f(z)z$ belongs to $\mathcal{S}_\alpha^*(\mathbb{B})$.
 Then, for every $z\in\mathbb{B}\setminus\{0\}$, we have
 \begin{align}\label{Th-C}
 \displaystyle
 \left|\left(\frac{l_{z}(D^4F(0)(z^4))}{4!\|z\|^4}\right)^2-\left(\frac{l_z(D^3F(0)(z^3))}{3!\|z\|^3}\right)^2\right|\le \frac{1}{9}(1-\alpha)^2(3-2\alpha)^2(25-16\alpha+4\alpha^2).
 	\end{align}
 	
 The inequality \eqref{Th-C} is sharp.
 \end{theo}
 \begin{cor}\label{Cor-recover}
 	Taking $X=\mathbb{C}$ (so that $\|z\|=|z|$, $l_z(w)=w$) and $\alpha=0$ 	in Theorem \ref{Th-1}, and writing $F(z)=zf(z)=z+\sum_{n\ge2}a_nz^n$, 	the bound \eqref{Th-C} reduces to
 	 	\begin{align*}
 		|a_3^2-a_4^2|=|T_2(3)(f)|\le 25,
 	\end{align*}
 	which is precisely the sharp bound of Ali et al. \emph{\cite{Ali et al.}}
 	for starlike functions of order zero.
 \end{cor}

 \begin{theo}\label{Th-2}Let $f\in\mathcal{H}(\mathbb{U}^n,\mathbb{C})$ with $f(0)=1$ and $F(z)=zf(z)\in\mathcal{S}^*_{\alpha}(\mathbb{U}^n)$. Then
 	\begin{align}\label{K-1}
 	\displaystyle
 	&\left\|\frac{1}{3!}D^3F(0)\left(z^2,\frac{D^3F(0)(z^3)}{3!}\right)-\frac{1}{4!}D^4F(0)\left(z^3,\frac{D^4F(0)(z^4)}{4!}\right)\right\|_{\infty}\\\leq&
 	\frac{1}{9}(1-\alpha)^2(3-2\alpha)^2(25-16\alpha+4\alpha^2)||z||_{\infty}\nonumber
 	\end{align}
 	for all $z\in\mathbb{U}^n$. The inequality \eqref{K-1} is sharp.
 \end{theo}
  \begin{theo}\label{Th-3} Let $f\in\mathcal{H}(\mathbb{B},\mathbb{C})$ satisfy $f(0)=1$ and suppose that the mapping $F:\mathbb{B}\to X$ defined by 	$F(z)=f(z)z$ belongs to $\mathcal{S}_\alpha^*(\mathbb{B})$ for $0\le\alpha\le\frac{2}{3}$.
 Then, for every $z\in\mathbb{B}\setminus\{0\}$, we have
 \begin{align}\label{Th-k}
 \displaystyle
 &\bigg|\left(\frac{l_{z}(D^2F(0)(z^2))}{2!\|z\|^2}-\frac{l_z(D^4F(0)(z^4))}{4!\|z\|^4}\right)\bigg(\left(\frac{l_{z}(D^2F(0)(z^2))}{2!\|z\|^2}\right)^2+\nonumber\\&\frac{l_{z}(D^2F(0)(z^2))}{2!\|z\|^2}\frac{l_{z}(D^4F(0)(z^4))}{4!\|z\|^4}-2\left(\frac{l_{z}(D^3F(0)(z^3))}{3!\|z\|^3}\right)^2\bigg)\bigg|\nonumber\\\le& \frac{4}{9}(1-\alpha)^3(9-7\alpha+2\alpha^2)(21-22\alpha+8\alpha^2).
 \end{align}
 
 The inequality \eqref{Th-k} is sharp.
 \end{theo}
 \begin{cor}\label{Cor-recover2}
 	Similarly, taking $X=\mathbb{C}$ and $\alpha=0$ in Theorem \ref{Th-3}, the bound \eqref{Th-k} reduces to
 	\begin{align*}
 		|(a_2-a_4)(a_2^2+a_2a_4-2a_3^2)|=|T_3(2)(f)|\le 84,
 	\end{align*}
 	recovering the sharp bound of Ali et al. \emph{\cite{Ali et al.}.}
 \end{cor}

 \begin{theo}\label{Th-4}Let $f\in\mathcal{H}(\mathbb{U}^n,\mathbb{C})$ with $f(0)=1$ and $F(z)=zf(z)\in\mathcal{S}^*_{\alpha}(\mathbb{U}^n)$ for $0\le\alpha\le\frac{2}{3}$. Then
 	\begin{align}\label{K-2}
 		\displaystyle
 		&\bigg\|\bigg(\frac{1}{2!}D^2F(0)\left(z,\frac{D^2F(0)(z^2)}{2!}\right)-\frac{1}{4!}D^4F(0)\left(z^3,\frac{D^2F(0)(z^2)}{2!}\right)\nonumber\\&+\frac{2}{3!}D^3F(0)\left(z^2,\frac{D^3F(0)(z^3)}{3!}\right)\bigg)\left(\frac{1}{2!}D^2F(0)(z,z)-\frac{1}{4!}D^4F(0)(z,z^3)\right)\bigg\|_{\infty}\nonumber\\\leq&
 		\frac{4}{9}(1-\alpha)^3(9-7\alpha+2\alpha^2)(21-22\alpha+8\alpha^2)||z||_{\infty}
 	\end{align}
 	for all $z\in\mathbb{U}^n$. The inequality \eqref{K-2} is sharp.
 \end{theo}
 \begin{rem}\label{rem:B}
 	Set $	B_1(\alpha)=\tfrac19(1-\alpha)^2(3-2\alpha)^2(25-16\alpha+4\alpha^2)$, $ 	B_2(\alpha)=\tfrac49(1-\alpha)^3(9-7\alpha+2\alpha^2)(21-22\alpha+8\alpha^2).$ 	Then Theorems~\ref{Th-1}--\ref{Th-2} give the sharp bound $B_1(\alpha)$ for
 	$0\le\alpha<1$, and Theorems~\ref{Th-3}--\ref{Th-4} give $B_2(\alpha)$ for
 	$0\le\alpha\le\frac23$. In particular $B_1(0)=25$ and $B_2(0)=84$;
 	see Figure~\ref{fig:bounds}.
 \end{rem}
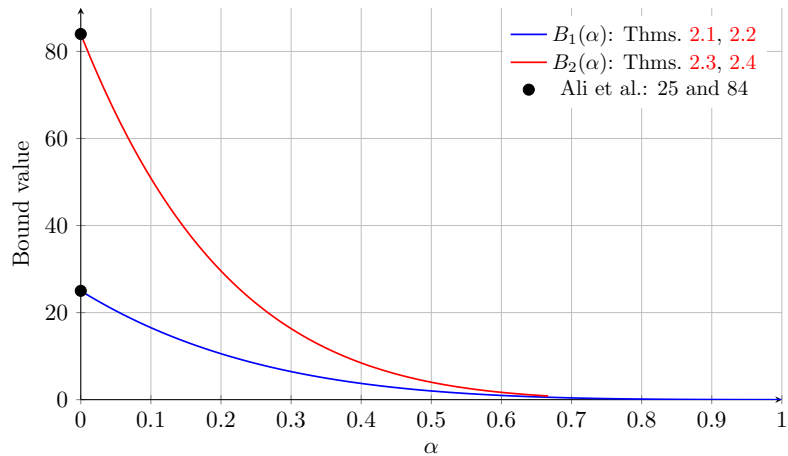
\begin{figure}[htbp]
	\centering
	\scalebox{0.8}{%
		\begin{tikzpicture}
			\begin{axis}[
				width=0.85\textwidth, height=0.52\textwidth,
				xlabel={$\alpha$}, ylabel={Bound value},
				xmin=0, xmax=1, ymin=0, ymax=90,
				grid=major, axis lines=left, samples=200,
				legend style={
					at={(0.98,0.98)},
					anchor=north east,
					draw=none,
					font=\small
				}]
				
				\addplot[blue,thick,domain=0:0.999]
				{(1/9)*(1-x)^2*(3-2*x)^2*(25-16*x+4*x^2)};
				\addlegendentry{$B_1(\alpha)$: Thms.~\ref{Th-1}, \ref{Th-2}}
				
				\addplot[red,thick,domain=0:0.6667]
				{(4/9)*(1-x)^3*(9-7*x+2*x^2)*(21-22*x+8*x^2)};
				\addlegendentry{$B_2(\alpha)$: Thms.~\ref{Th-3}, \ref{Th-4}}
				
				\addplot[
				only marks,
				mark=*,
				mark size=2.6pt,
				black
				] coordinates {(0,25)(0,84)};
				\addlegendentry{Ali et al.: $25$ and $84$}
				
			\end{axis}
		\end{tikzpicture}%
	}
	\caption{The sharp bounds $B_1(\alpha)$ (for $0\le\alpha<1$) and	$B_2(\alpha)$ (for $0\le\alpha\le2/3$). The dots mark the constants
		$25$ and $84$ of Ali et al.\ recovered at $\alpha=0$.}
	\label{fig:bounds}
\end{figure}
 \subsection{{Auxiliary lemmas}}
 Let $\mathcal{P}$ denote the class of analytic functions $h$ in the open unit disk $\mathbb{U}$ satisfying $h(0) = 1$ and $\Re\{h(z)\}> 0$ for all $z \in \mathbb{U}$. A member of $\mathcal P$ is called a Carath\'eodory function.  The following lemma is instrumental in establishing our main results.
\begin{lem}\label{Lem1}\emph{\cite{Cara, Pomme}}
Let $h(z)=1+\sum_{n=1}^{\infty}h_nz^n\in\mathcal{P}$. Then, the following
	inequalities hold true:
	\begin{align*}
		|h_n| &\leq 2, \qquad n \geq 1,\\
		|h_{n+2k}-\mu h_n h_k^2| &\leq 2(1+2\mu),
		\qquad 0\leq\mu\leq 1,\\
		\left|h_2-\frac{h_1^2}{2}\right|
		&\leq 2-\frac{|h_1|^2}{2},
	\end{align*}
	and $$|h_{n+k}-\mu h_nh_k|\leq 2\max\{1,|2\mu-1|\}.$$
\end{lem}

\begin{proof}[\bf Proof of Theorem \ref{Th-1}]
By the given condition, we have $F\in \mathcal{S}^*_{\alpha}(\mathbb{B})$, where $0\leq \alpha<1$. Then \eqref{SL1} yields
\begin{align*}
	\Re\left\lbrace \frac{\|z\|}{l_z\left([D F(z)]^{-1}F(z)\right)}\right\rbrace>\alpha, \quad z\in\mathbb{B}\backslash \{0\},\quad l_z\in T(z).
\end{align*}

Let $z\in \mathbb{B}\backslash\{0\}$ be fixed and let $z_0=z/\|z\|$. The function $g:\mathbb{U}\to \mathbb{C}$ is defined as follows
\begin{align}\label{Th-21.1}	
	g(\zeta)=\begin{cases}
		\frac{\zeta}{l_{z_0}\left([D F(\zeta z_0)]^{-1}F(\zeta z_0)\right)},& \zeta\neq 0,\\[8pt]
		1& \zeta=0.
	\end{cases}
\end{align}
Clearly, $g\in\mathcal{H}(\mathbb{U})$ and $g(0)=1$.
Since $l_{\zeta z}(.)=\frac{|\zeta|}{\zeta}l_z$, for each $\zeta\in\mathbb{C}\backslash\{0\}$, we have
\begin{align*}
	g(\zeta)=\frac{\zeta}{l_{z_0}\left([D F(\zeta z_0)]^{-1}F(\zeta z_0)\right)}=\frac{\|\zeta z_0\|}{l_{\zeta z_0}\left([D F(\zeta z_0)]^{-1}F(\zeta z_0)\right)}
\end{align*}	
for $\zeta\in\mathbb{U}\backslash\{0\}$ and hence  $\Re\{g(\zeta)\}>\alpha$
for all $\zeta\in\mathbb{U}$. 
\par Let $h(\zeta)=\frac{g(\zeta)-\alpha}{1-\alpha}$ for all $\zeta\in\mathbb{U}$. We note that  $h\in\mathcal{H}(\mathbb{U})$ with $\Re(h(\zeta))>0$ for all $\zeta\in\mathbb{U}$ and $h(0)=1$. Also let
$h(\zeta)=1+h_1\zeta+h_2\zeta^2+h_3\zeta^3+\ldots$. It follows that,
\begin{align}\label{Th-21.2}
	g(\zeta)=1+(1-\alpha)h_1\zeta+(1-\alpha)h_2\zeta^2+(1-\alpha)h_3\zeta^3+\ldots
\end{align}	

Using a similar method as in Theorem 7.1.14 \cite{Graham-Kohr}, we have
\begin{align}\label{Th-21.2b}
	[D F(z)]^{-1}\eta=\frac{1}{f(z)}\left\lbrack \eta-\frac{\frac{(Df(z)\eta)z}{f(z)}}{1+\frac{Df(z)z}{f(z)}}\right\rbrack.
\end{align}

Therefore 
\begin{align}\label{Th-21.3}
	[D F(z)]^{-1}F(z)=\frac{zf(z)}{f(z)+D f(z)z},
\end{align}
which implies that $\frac{\|z\|}{l_z\left([D F(z)]^{-1}F(z)\right)}=1+\frac{Df(z) z}{f(z)}$
and so
\begin{align}\label{Th-21.4}
	g(\zeta)=\frac{\|\zeta z_0\|}{l_{\zeta z_0}\left([D F(\zeta z_0)]^{-1}F(\zeta z_0)\right)}=1+\frac{Df(\zeta z_0) \zeta z_0}{f(\zeta z_0)}.
\end{align}

Let
\begin{align}\label{Th-21.5}
	f(\zeta z_0)=1+D f(0)(z_0)\zeta+\frac{1}{2!}D^2 f(0)(z_0^2)\zeta^2+\frac{1}{3!}D^3 f(0)(z_0^3)\zeta^3+\ldots.
\end{align}

Note that 
\begin{align}\label{Th-21.6}
	D f(\zeta z_0)(\zeta z_0)=D f(0)(z_0)\zeta+D^2 f(0)(z_0^2)\zeta^2+\frac{1}{2}D^3 f(0)(z_0^3)\zeta^3+\ldots.
\end{align}

Using \eqref{Th-21.5} and \eqref{Th-21.6} in \eqref{Th-21.4}, we obtain
\begin{align}\label{Th-21.7}
	&\left(1+(1-\alpha)h_1\zeta+(1-\alpha)h_2\zeta^2+(1-\alpha)h_3\zeta^3+\ldots\right)\times\\&\left(1+D f(0)(z_0)\zeta+\frac{1}{2!}D^2 f(0)(z_0^2)\zeta^2+\frac{1}{3!}D^3 f(0)(z_0^3)\zeta^3+\ldots\right)\nonumber\\=&
	1+2Df(0)(z_0)\zeta+\frac{3}{2}D^2f(0)(z_0^2)\zeta^2+\frac{2}{3}D^3f(0)(z_0^3)\zeta^3+\ldots.\nonumber
\end{align}

From \eqref{Th-21.7}, we find that
\begin{align}\label{Th-21.8}
	\begin{cases}
		Df(0)(z_0)&=(1-\alpha)h_1,\\[8pt]
		D^2f(0)(z_0^2)&=(1-\alpha)h_2+((1-\alpha)h_1)^2,\\[8pt]
		D^3f(0)(z_0^3)&=\left((1-\alpha)h_1\right)^3+3(1-\alpha)^2h_1h_2+2(1-\alpha)h_3.
	\end{cases}
\end{align}

Since $F(z)=zf(z)$, we have (see \cite[Theorem 2]{Xiong-Yu-Wang-2025})
\begin{align}\label{Th-2x}
	D^mF(0)(z^m)=mD^{m-1}f(0)(z^{m-1})z,
\end{align}
where $m\in\mathbb{N}$.
Consequently from \eqref{Th-2x} it follows that
\begin{align}\label{Th-2y}
	\displaystyle
	\frac{l_{z}(D^mF(0)(z^m))}{m!}=\frac{D^{m-1}f(0)(z^{m-1})}{(m-1)!}||z||.
\end{align}

In particular, we have 
\begin{align}\label{Th-1.8}
	\begin{cases}
		\frac{l_{z}(D^4F(0)(z^4))}{4!}&=\frac{D^3f(0)(z^3)}{3!}\|z\|,\\[8pt]
		\frac{l_{z}(D^3F(0)(z^3))}{3!}&=\frac{D^2f(0)(z^2)}{2!}\|z\|,\\[8pt]
		\frac{l_z(D^2F(0)(z^2))}{2!}&=Df(0)(z) \|z\|.
	\end{cases}
\end{align}

Using \eqref{Th-1.8} and \eqref{Th-21.8}, we get
\begin{align}\label{Th-21.9}
	&\frac{l_{z}(D^4F(0)(z^4))}{4!\|z\|^4}=\frac{1}{6}\left(\left((1-\alpha)h_1\right)^3+3(1-\alpha)^2h_1h_2+2(1-\alpha)h_3\right),
\end{align}
\begin{align}\label{Th-21.10}
	\frac{l_{z}(D^3F(0)(z^3))}{3!\|z\|^3}=\frac{1}{2}\left((1-\alpha)h_2+((1-\alpha)h_1)^2\right)
\end{align}
and
\begin{align}\label{Th-21.11}
	\frac{l_z(D^2F(0)(z^2))}{2!\|z\|^2}=(1-\alpha)h_1.
\end{align}

\begin{align}\label{Th-3.2.19}
	\displaystyle
	&\left|\left(\frac{l_{z}(D^4F(0)(z^4))}{4!\|z\|^4}\right)^2-\left(\frac{l_z(D^3F(0)(z^3))}{3!\|z\|^3}\right)^2\right|\nonumber\\&\leq\left|\frac{l_{z}(D^3F(0)(z^3))}{3!\|z\|^3}\right|^2+\left|\frac{l_z(D^4F(0)(z^4))}{4!\|z\|^4}\right|^2\nonumber\\&\le\frac{(1-\alpha)^2}{4}\left(\frac{1}{9}\left((1-\alpha)^2|h_1|^3+3(1-\alpha)|h_1||h_2|+2|h_3|\right)^2+\left(|h_2|+(1-\alpha)|h_1|^2\right)^2\right). 
\end{align}

Therefore, applying Lemma \ref{Lem1} to \eqref{Th-3.2.19}, we obtain 
\begin{align*}
	\displaystyle
	\left|\left(\frac{l_{z}(D^4F(0)(z^4))}{4!\|z\|^4}\right)^2-\left(\frac{l_z(D^3F(0)(z^3))}{3!\|z\|^3}\right)^2\right|\le \frac{1}{9}(1-\alpha)^2(3-2\alpha)^2(25-16\alpha+4\alpha^2).
\end{align*}

\medskip	
To prove sharpness, consider the function $F$ defined by
\begin{align}\label{Sh-1}
	\displaystyle
	F(z)=zf(z)=\frac{z}{\left(1-il_u(z)\right)^{2(1-\alpha)}},\quad z\in\mathbb{B},\quad \|u\|=1.
\end{align}

According to Lemma 2.2 \cite{Xu-He-Xu}, it is clear that $F\in\mathcal{S}^*_{\alpha}(\mathbb{B})$. Note that
\begin{align*}
	f(z)=&1+i2(1-\alpha)l_u(z)-\frac{2(1-\alpha)(3-2\alpha)}{2!}(l_u(z))^2\\&-i\frac{2(1-\alpha)(3-2\alpha)(4-2\alpha)}{3!}(l_u(z))^3+\ldots.
\end{align*}

Since $F(z)=zf(z)$, we have
\begin{align*}
	F(z)=&z+2i(1-\alpha)l_u(z)z-\frac{2(1-\alpha)(3-2\alpha)}{2!}(l_u(z))^2z\\&-\frac{2i(1-\alpha)(3-2\alpha)(4-2\alpha)}{3!}(l_u(z))^3z+\ldots
\end{align*}

Consequently, we have
\begin{align*}
	\displaystyle
	&\frac{D^3F(0)(z^3)}{3!}=-(1-\alpha)(3-2\alpha)(l_u(z))^2z,\\
	\displaystyle &\frac{D^4F(0)(z^4)}{4!}=-\frac{2i}{3}(1-\alpha)(3-2\alpha)(2-\alpha)(l_u(z))^3z.
\end{align*}
Therefore, using \eqref{Th-1.8} we get
\begin{align*}
	\displaystyle
	&\frac{l_z(D^3F(0)(z^3))}{3!}=-(1-\alpha)(3-2\alpha)(l_u(z))^2\|z\|,\\
	\displaystyle
	&\frac{l_z(D^4F(0)(z^4))}{4!}=-\frac{2i}{3}(1-\alpha)(3-2\alpha)(2-\alpha)(l_u(z))^3\|z\|.
\end{align*}
Taking $z=ru\;(0<r<1)$, we get
\begin{align*}
	\displaystyle
	\frac{l_z(D^3F(0)(z^3))}{3!\|z\|^3}=-(1-\alpha)(3-2\alpha)\;\; \text{and}\quad
	\frac{l_z(D^4F(0)(z^4))}{4!\|z\|^4}=-\frac{2i}{3}(1-\alpha)(3-2\alpha)(2-\alpha).
\end{align*}
Hence
\begin{align*}
	\displaystyle
	\left|\left(\frac{l_{z}(D^4F(0)(z^4))}{4!\|z\|^4}\right)^2-\left(\frac{l_z(D^3F(0)(z^3))}{3!\|z\|^3}\right)^2\right|= \frac{1}{9}(1-\alpha)^2(3-2\alpha)^2(25-16\alpha+4\alpha^2),
\end{align*}
  Figure~\ref{fig:sharpness-rect} shows geometrically why \eqref{Th-C} is sharp and Figure~\ref{fig:extremal-image} shows the extremal mapping \eqref{Sh-1} restricted to the complex line through $u$.
\end{proof}

\begin{figure}[H]
	\centering
	\begin{tikzpicture}
		\begin{axis}[
			width=0.70\textwidth,
			height=0.58\textwidth,
			axis equal image,
			xmin=0, xmax=5.6,
			ymin=0, ymax=5.6,
			xlabel={$|a_3|$},
			ylabel={$|a_4|$},
			%
			axis background/.style={
				fill=blue!2
			},
			grid=major,
			grid style={
				densely dotted,
				gray!35,
				line width=0.5pt
			},
			minor grid style={
				dotted,
				gray!15
			},
			minor tick num=1,
			%
			axis lines=left,
			axis line style={
				black,
				line width=0.9pt,
				-{Stealth[length=2.2mm]}
			},
			tick style={
				black,
				line width=0.7pt
			},
			%
			xlabel style={
				font=\large,
				xshift=4pt
			},
			ylabel style={
				font=\large,
				yshift=4pt
			},
			tick label style={
				font=\small
			},
			%
			clip=false,
			%
			legend style={
				at={(1.03,0.98)},
				anchor=north west,
				font=\small,
				draw=blue!45!black,
				fill=white,
				fill opacity=0.94,
				text opacity=1,
				rounded corners=2pt,
				line width=0.7pt,
				inner xsep=6pt,
				inner ysep=5pt,
				row sep=2pt
			},
			legend cell align=left
			]

			\addplot[
			blue!75!black,
			very thick
			]
			coordinates {
				(0,0) (3,0) (3,4) (0,4) (0,0)
			};
			\addlegendentry{
				$\alpha=0:\quad
				|a_3|\leq3,\ |a_4|\leq4$
			}

			\addplot[
			blue!75!black,
			dashed,
			very thick,
			domain=0:90,
			samples=120
			]
			({5*cos(x)},{5*sin(x)});
			\addlegendentry{
				$|a_3|^2+|a_4|^2=25$
			}

			\addplot[
			red!75!black,
			very thick
			]
			coordinates {
				(0,0) (1,0) (1,1) (0,1) (0,0)
			};
			\addlegendentry{
				$\alpha=\frac12:\quad
				|a_3|\leq1,\ |a_4|\leq1$
			}

			\addplot[
			red!75!black,
			dashed,
			very thick,
			domain=0:90,
			samples=120
			]
			({sqrt(2)*cos(x)},{sqrt(2)*sin(x)});
			\addlegendentry{
				$|a_3|^2+|a_4|^2=2$
			}

			\addplot[
			only marks,
			mark=*,
			mark size=2.6pt,
			black
			]
			coordinates {
				(3,4) (1,1)
			};

			\node[
			fill=white,
			fill opacity=0.92,
			text opacity=1,
			rounded corners=1pt,
			inner sep=2pt,
			font=\scriptsize,
			anchor=south west
			]
			at (axis cs:3,4)
			{$(3,4)$};

			\node[
			fill=white,
			fill opacity=0.92,
			text opacity=1,
			rounded corners=1pt,
			inner sep=2pt,
			font=\scriptsize,
			anchor=south west
			]
			at (axis cs:1,1)
			{$(1,1)$};

			\node[
			fill=white,
			fill opacity=0.90,
			text opacity=1,
			rounded corners=1pt,
			inner sep=2.5pt,
			font=\scriptsize,
			anchor=south east
			]
			at (axis cs:3.95,4.65)
			{$|a_3^2-a_4^2|
				=|a_3|^2+|a_4|^2$};
			
		\end{axis}
	\end{tikzpicture}
	
	\caption{Coefficient-plane illustration of the sharpness in	Theorem~\ref{Th-1}. The bounds
		$|a_3|\le(1-\alpha)(3-2\alpha)$ and
		$|a_4|\le\frac{2}{3}(1-\alpha)(3-2\alpha)(2-\alpha)$
		are shown for $\alpha=0$ and $\alpha=\frac12$.
		}
	\label{fig:sharpness-rect}
\end{figure}
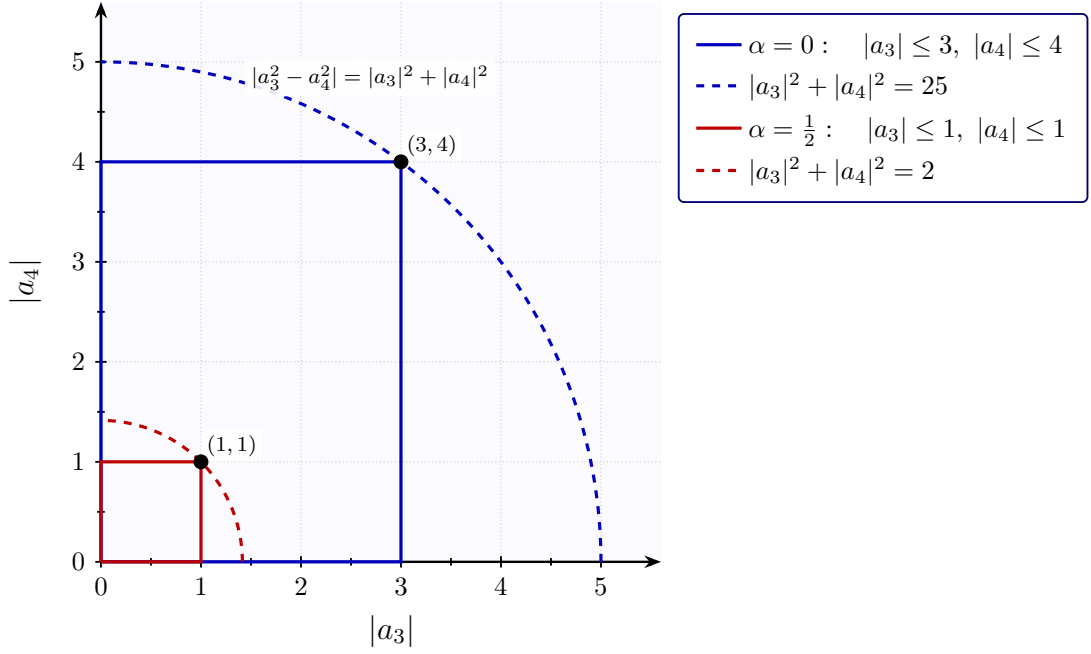

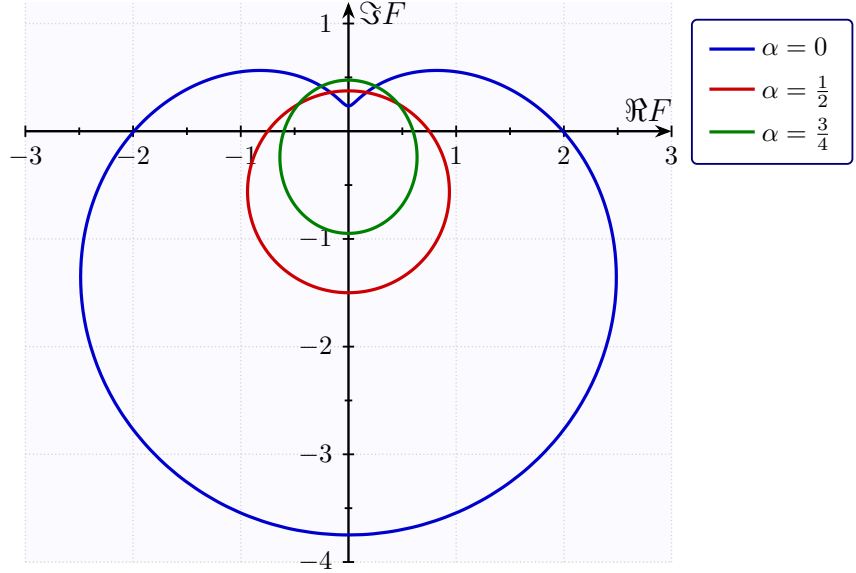
\begin{figure}[H]
	\centering
	\begin{tikzpicture}
		\begin{axis}[
			width=0.70\textwidth,
			height=0.58\textwidth,
			axis equal image,
			xlabel={$\Re F$},
			ylabel={$\Im F$},
			xmin=-3, xmax=3,
			ymin=-4, ymax=1.2,
			axis lines=middle,
			samples=400,
			domain=0:360,
			variable=t,
			%
			axis background/.style={
				fill=blue!2
			},
			grid=major,
			grid style={
				densely dotted,
				gray!35,
				line width=0.5pt
			},
			minor grid style={
				dotted,
				gray!15
			},
			minor tick num=1,
			%
			axis line style={
				black,
				line width=0.9pt,
				-{Stealth[length=2.2mm]}
			},
			tick style={
				black,
				line width=0.7pt
			},
			%
			xlabel style={
				font=\large,
				xshift=4pt
			},
			ylabel style={
				font=\large,
				yshift=4pt
			},
			tick label style={
				font=\small
			},
			%
			clip=false,
			%
			legend style={
				at={(1.03,0.97)},
				anchor=north west,
				font=\small,
				draw=blue!45!black,
				fill=white,
				fill opacity=0.94,
				text opacity=1,
				rounded corners=2pt,
				line width=0.7pt,
				inner xsep=6pt,
				inner ysep=5pt,
				row sep=2pt
			},
			legend cell align=left
			]
			
			\addplot[
			blue!80!black,
			very thick,
			smooth
			]
			({0.6*cos(t-2*atan2(-0.6*cos(t),1+0.6*sin(t)))
				/(1.36+1.2*sin(t))},
			{0.6*sin(t-2*atan2(-0.6*cos(t),1+0.6*sin(t)))
				/(1.36+1.2*sin(t))});
			\addlegendentry{$\alpha=0$}
			
			\addplot[
			red!80!black,
			very thick,
			smooth
			]
			({0.6*cos(t-atan2(-0.6*cos(t),1+0.6*sin(t)))
				/(1.36+1.2*sin(t))^0.5},
			{0.6*sin(t-atan2(-0.6*cos(t),1+0.6*sin(t)))
				/(1.36+1.2*sin(t))^0.5});
			\addlegendentry{$\alpha=\tfrac12$}
			
			\addplot[
			green!50!black,
			very thick,
			smooth
			]
			({0.6*cos(t-0.5*atan2(-0.6*cos(t),1+0.6*sin(t)))
				/(1.36+1.2*sin(t))^0.25},
			{0.6*sin(t-0.5*atan2(-0.6*cos(t),1+0.6*sin(t)))
				/(1.36+1.2*sin(t))^0.25});
			\addlegendentry{$\alpha=\tfrac34$}
			
		\end{axis}
	\end{tikzpicture}
	
	\caption{Image of the circle $|z|=0.6$ under the extremal function
		$F(z)=z(1-iz)^{-2(1-\alpha)}$ of~\eqref{Sh-1}, for three values of $\alpha$.
		The images are starlike curves that become more elongated as $\alpha$ decreases.}
	\label{fig:extremal-image}
\end{figure}
\begin{proof}[\bf Proof of Theorem \ref{Th-2}]
By the given condition, $F\in\mathcal{S}^*_{\alpha}(\mathbb{U}^n)$, where $0\leq \alpha<1$. Therefore
\begin{align}\label{Th-22.1}
	\displaystyle
	\Re\left\lbrace \frac{\|z\|_{\infty}}{l_z\left([D F(z)]^{-1}F(z)\right)}\right\rbrace>\alpha, \quad z\in\mathbb{U}^n\backslash \{0\},\quad l_z\in T(z).
\end{align}

Let $w(z)=(w_1(z),w_2(z),\ldots,w_n(z))'=[D F(z)]^{-1}F(z)$ be a column vector in $\mathbb{C}^n$. According to (\ref{eq3}), we have
\begin{align}\label{Th-22.2}
	\displaystyle
	l_z(w)=\sum\limits_{\{j:||z||_{\infty}=|z_j|\}}t_j\frac{w_j(z)\ol z_j}{||z||_{\infty}},
\end{align}
where each $t_j\geq 0$ and $\sum\limits_{\{j:||z||_{\infty}=|z_j|\}}t_j=1.$
\par Let $k\in\{1,2,\ldots,n\}$ be such that $||z||_{\infty}=|z_k|$. Suppose $t_k=1$ and $t_j=0$, where $j\in\{1,2,\ldots,n\}$ such that $j\neq k$. Then from \eqref{Th-22.1} and \eqref{Th-22.2}, we have
\begin{align}\label{Th-22.3}
	\displaystyle
	\Re\left\lbrace \frac{z_k}{w_k(z)}\right\rbrace>\alpha,\quad z\in\mathbb{U}^n\backslash\{0\},
\end{align}
where $w_k(z)$ is the $k$-th coordinate of $w(z)$ such that $||z||_{\infty}=|z_k|$.

Let $z\in\mathbb{U}^n\backslash\{0\}$ and let $z_0=z/\|z\|$. We define the function $g_k:\mathbb{U}\to \mathbb{C}$ as follows
\begin{align}\label{Th-22.4}	
	\displaystyle
	g_k(\zeta)=\begin{cases}
		\frac{\zeta z_k}{w_k(\zeta z_0)},& \zeta\neq 0,\\[8pt]
		1& \zeta=0.
	\end{cases}
\end{align}

It follows from \eqref{Th-22.3} and \eqref{Th-22.4} that $\Re\{g_k(\zeta)\}>0$ for all $\zeta\in\mathbb{U}$. 
Let
\begin{align}\label{Th-22.5}
	\displaystyle
	h_k(\zeta)=\frac{g_k(\zeta)-\alpha}{1-\alpha}
\end{align}
for all $\zeta\in\mathbb{U}$. Since $\Re\{g_k(\zeta)\}>\alpha$ 
for all $\zeta\in\mathbb{U}$, we deduce from \eqref{Th-22.5} that $h_k\in\mathcal{H}(\mathbb{U})$ with $\Re\{h_k(\zeta)\}>0$ for all $\zeta\in\mathbb{U}$ and $h_k(0)=1$.
On the other hand, from \eqref{Th-21.3}, we have
\begin{align}\label{Th-22.6}
	\displaystyle
	w_k(\zeta z_0)=\zeta z_k\left\lbrack \frac{f(\zeta z_0)}{f(\zeta z_0)+D f(\zeta z_0)\zeta z_0}\right\rbrack.
\end{align}

Using \eqref{Th-22.4} and \eqref{Th-22.6} to \eqref{Th-22.5}, we get
\begin{align}\label{Th-22.7}
	\displaystyle
	\left((1-\alpha)h_k(\zeta)+\alpha\right)f(\zeta z_0)=f(\zeta z_0)+D f(\zeta z_0)\zeta z_0.	
\end{align} 	

Let
\begin{align}\label{Th-22.7a}
	h_k(\zeta)=1+h_{k1}\zeta+h_{k2}\zeta^2+h_{k3}\zeta^3+\ldots
\end{align}

Now using \eqref{Th-21.5}, \eqref{Th-21.6} and \eqref{Th-22.7a} to \eqref{Th-22.7} and then comparing coefficients, we obtain
\begin{align}\label{Th-22.8}
	\begin{cases}
		Df(0)(z_0)&=(1-\alpha)h_{k1},\\[8pt]
		D^2f(0)(z_0^2)&=(1-\alpha)h_{k2}+((1-\alpha)h_{k1})^2,\\[8pt]
		D^3f(0)(z_0^3)&=\left((1-\alpha)h_{k1}\right)^3+3(1-\alpha)^2h_{k1}h_{k2}+2(1-\alpha)h_{k3}.
	\end{cases}
\end{align}

Let $F(z)=(F_1(z),F_2(z),\ldots,F_n(z))$. Since $F(z)=zf(z)$, it follows that
\begin{align*}
	\displaystyle
	\frac{D^mF_k(0)(z^m)}{m!}=\frac{D^{m-1}f(0)(z^{m-1})}{(m-1)!}z_k.
\end{align*}

Since $z_0=z/\|z\|$, we have 
\begin{align}\label{Th-22.11}
	\displaystyle
	\frac{D^mF_k(0)(z_0^m)}{m!}=\frac{D^{m-1}f(0)(z_0^{m-1})}{(m-1)!}\frac{z_k}{||z||_{\infty}}.
\end{align}

In view of \eqref{Th-2x} and \eqref{Th-22.11}, we obtain
\begin{align}\label{T-22.12}
	\displaystyle
	&\left|\frac{1}{3!}D^3F_k(0)\left(z_0^2,\frac{D^3F(0)(z_0^3)}{3!}\right)\frac{||z||_{\infty}}{z_k}-\frac{1}{4!}D^4F_k(0)\left(z_0^3,\frac{D^4F(0)(z_0^4)}{4!}\right)\frac{||z||_{\infty}}{z_k}\right|\nonumber\\=&
	\left|\frac{1}{3!}D^3F_k(0)\left(z_0^2,\frac{D^2f(0)(z_0^2)z_0}{2!}\right)\frac{||z||_{\infty}}{z_k}- \frac{1}{4!}D^4F_k(0)\left(z_0^3,\frac{D^3f(0)(z_0^3)z_0}{3!}\right)\frac{||z||_{\infty}}{z_k}\right|\nonumber\\=&
	\left|\frac{1}{2!} D^2f(0)(z_0^2)\frac{1}{3!}D^3F_k(0)\left(z_0^2,z_0\right)\frac{||z||_{\infty}}{z_k}- \frac{D^3f(0)(z_0^3)}{3!}\frac{1}{4!}D^4F_k(0)(z_0^3,z_0)\frac{||z||_{\infty}}{z_k}\right|
	\nonumber\\=&
	\left| \frac{1}{2!}D^2f(0)(z_0^2)\frac{1}{3!}D^3F_k(0)\left(z_0^3\right)\frac{||z||_{\infty}}{z_k}-\frac{D^3f(0)(z_0^3)}{3!}\frac{1}{4!}D^4F_k(0)(z_0^4)\frac{||z||_{\infty}}{z_k}\right|
	\nonumber\\=&
	\left|\left(\frac{D^2f(0)(z_0^2)}{2!}\right)^2- \left(\frac{D^3f(0)(z_0^3)}{3!}\right)^2\right|.
\end{align}

Using \eqref{Th-22.8} to \eqref{T-22.12}, we get
\begin{align*}
	\displaystyle
	&\left|\frac{1}{3!}D^3F_k(0)\left(z_0^2,\frac{D^3F(0)(z_0^3)}{3!}\right)\frac{||z||_{\infty}}{z_k}-\frac{1}{4!}D^4F_k(0)\left(z_0^3,\frac{D^4F(0)(z_0^4)}{4!}\right)\frac{||z||_{\infty}}{z_k}\right|\\ &\le\frac{(1-\alpha)^2}{4}\left(\frac{1}{9}\left((1-\alpha)^2|h_{k1}|^3+3(1-\alpha)|h_{k1}||h_{k2}|+2|h_{k3}|\right)^2+\left(|h_{k2}|+(1-\alpha)|h_{k1}|^2\right)^2\right).
	\nonumber
\end{align*}

Therefore, by the argument used in the proof of Theorem~\ref{Th-1}, we obtain 
\begin{align*}
	\displaystyle
	&\left\|\frac{1}{3!}D^3F(0)\left(z^2,\frac{D^3F(0)(z^3)}{3!}\right)-\frac{1}{4!}D^4F(0)\left(z^3,\frac{D^4F(0)(z^4)}{4!}\right)\right\|_{\infty}\\\leq&
	\frac{1}{9}(1-\alpha)^2(3-2\alpha)^2(25-16\alpha+4\alpha^2)||z||_{\infty}.
\end{align*}

In particular, when $z_0\in\partial_0 \mathbb{U}^n$, we deduce that
\begin{align}\label{Th-22.13}
	\displaystyle
	&\left|\frac{1}{3!}D^3F_k(0)\left(z_0^2,\frac{D^3F(0)(z_0^3)}{3!}\right)-\frac{1}{4!}D^4F_k(0)\left(z_0^3,\frac{D^4F(0)(z_0^4)}{4!}\right)\right|\\\leq& \frac{1}{9}(1-\alpha)^2(3-2\alpha)^2(25-16\alpha+4\alpha^2),	\quad k=1,2,\ldots,n.\nonumber
\end{align}

We know that
\begin{align*}
	\frac{1}{3!}D^3F_k(0)\left(z^2,\frac{D^3F(0)(z^3)}{3!}\right)-\frac{1}{4!}D^4F_k(0)\left(z^3,\frac{D^4F(0)(z^4)}{4!}\right)\in\mathcal{H}(\ol{\mathbb{U}^n}),
\end{align*}
$k=1,2,\ldots,n$.
By the maximum modulus theorem for holomorphic functions on the unit polydisk, we deduce from \eqref{Th-22.13} that
\begin{align*}
	\displaystyle
	&\left|\frac{1}{3!}D^3F_k(0)\left(z_0^2,\frac{D^3F(0)(z_0^3)}{3!}\right)-\frac{1}{4!}D^4F_k(0)\left(z_0^3,\frac{D^4F(0)(z_0^4)}{4!}\right)\right|\\\leq& 
	\frac{1}{9}(1-\alpha)^2(3-2\alpha)^2(25-16\alpha+4\alpha^2),
\end{align*}
for all $z_0\in\partial \mathbb{U}^n$, where $k=1,2,\ldots,n$.
Hence
\begin{align*}
\displaystyle
&\left|\frac{1}{3!}D^3F_k(0)\left(z_0^2,\frac{D^3F(0)(z_0^3)}{3!}\right)-\frac{1}{4!}D^4F_k(0)\left(z_0^3,\frac{D^4F(0)(z_0^4)}{4!}\right)\right|\\\leq& 
\frac{1}{9}(1-\alpha)^2(3-2\alpha)^2(25-16\alpha+4\alpha^2),
\end{align*}
for all $z\in\mathbb{U}^n$, where $k=1,2,\ldots,n$. Therefore
\begin{align*}
	\displaystyle
	&\left\|\frac{1}{3!}D^3F(0)\left(z^2,\frac{D^3F(0)(z^3)}{3!}\right)-\frac{1}{4!}D^4F(0)\left(z^3,\frac{D^4F(0)(z^4)}{4!}\right)\right\|_{\infty}\\\leq&
	\frac{1}{9}(1-\alpha)^2(3-2\alpha)^2(25-16\alpha+4\alpha^2)||z||_{\infty}, \quad z\in\mathbb{U}^n.
\end{align*}

\medskip
To establish the sharpness, let us consider the function $F$ defined by \begin{align}\label{Sh-2}
	F(z)=zf(z)=\left(\frac{z_1}{(1-iz_1)^{2(1-\alpha)}},\frac{z_2}{(1-iz_1)^{2(1-\alpha)}},\ldots,\frac{z_n}{(1-iz_1)^{2(1-\alpha)}}\right).\end{align}

It is clear that $F\in\mathcal{S}^*_{\alpha}(\mathbb{U}^n)$. 
Note that 
\begin{align*}
	F(z)=z+2i(1-\alpha)z_1z-(1-\alpha)(3-2\alpha)z_1^2z-\frac{2i(1-\alpha)(3-2\alpha)(2-\alpha)}{3}z_1^3z+\cdots,\;z\in\mathbb{U}^n.
\end{align*}
Now \eqref{Th-2x} yields
\begin{align*}
	\displaystyle
	\frac{D^3F(0)(z^3)}{3!}=-(1-\alpha)(3-2\alpha)z_1^2z\ \text{and}\
	\displaystyle \frac{D^4F(0)(z^4)}{4!}=-\frac{2i}{3}(1-\alpha)(3-2\alpha)(2-\alpha)z_1^3z.
\end{align*}
Taking $z=(r,0,\ldots,0)$, where $0<r<1$; one can easily verify that 
\begin{align*}
	\displaystyle
	&\left\|\frac{1}{3!}D^3F(0)\left(z^2,\frac{D^3F(0)(z^3)}{3!}\right)-\frac{1}{4!}D^4F(0)\left(z^3,\frac{D^4F(0)(z^4)}{4!}\right)\right\|_{\infty}\\=&
	\frac{1}{9}(1-\alpha)^2(3-2\alpha)^2(25-16\alpha+4\alpha^2)||z||_{\infty}, \quad z\in\mathbb{U}^n.
\end{align*}
which establishes the sharpness of the bound.
\end{proof}

\begin{proof}[\bf Proof of Theorem \ref{Th-3}]
Proceeding as in the proof of Theorem~\ref{Th-1}, we arrive at
\eqref{Th-21.9}--\eqref{Th-21.11}. It well known that
\begin{align}\label{eqn-2.37}
\displaystyle
&\bigg|\left(\frac{l_{z}(D^2F(0)(z^2))}{2!\|z\|^2}-\frac{l_z(D^4F(0)(z^4))}{4!\|z\|^4}\right)\bigg(\left(\frac{l_{z}(D^2F(0)(z^2))}{2!\|z\|^2}\right)^2\nonumber\\&+\frac{l_{z}(D^2F(0)(z^2))}{2!\|z\|^2}\frac{l_{z}(D^4F(0)(z^4))}{4!\|z\|^4}-2\left(\frac{l_{z}(D^3F(0)(z^3))}{3!\|z\|^3}\right)^2\bigg)\bigg|\nonumber\\=&\left|\frac{l_{z}(D^2F(0)(z^2))}{2!\|z\|^2}-\frac{l_z(D^4F(0)(z^4))}{4!\|z\|^4}\right|\bigg|\left(\frac{l_{z}(D^2F(0)(z^2))}{2!\|z\|^2}\right)^2\nonumber\\&+\frac{l_{z}(D^2F(0)(z^2))}{2!\|z\|^2}\frac{l_{z}(D^4F(0)(z^4))}{4!\|z\|^4}-2\left(\frac{l_{z}(D^3F(0)(z^3))}{3!\|z\|^3}\right)^2\bigg|=\mathcal{T}_1\mathcal{T}_2,
\end{align}
where \begin{align}
	\displaystyle\label{m-1}
	&\mathcal{T}_1=\left|\frac{l_{z}(D^2F(0)(z^2))}{2!\|z\|^2}-\frac{l_z(D^4F(0)(z^4))}{4!\|z\|^4}\right|,\\
	\displaystyle\label{m-2}
	&\mathcal{T}_2=\left|\left(\frac{l_{z}(D^2F(0)(z^2))}{2!\|z\|^2}\right)^2+\frac{l_{z}(D^2F(0)(z^2))}{2!\|z\|^2}\frac{l_{z}(D^4F(0)(z^4))}{4!\|z\|^4}-2\left(\frac{l_{z}(D^3F(0)(z^3))}{3!\|z\|^3}\right)^2\right|
	\end{align} 
\medskip

{\bf{A.}} Using \eqref{Th-21.9} and \eqref{Th-21.11} to \eqref{m-1}, we obtain
\begin{align}\label{eqn-2.38}
&\mathcal{T}_1\le\left|\frac{l_{z}(D^2F(0)(z^2))}{2!\|z\|^2}\right|+\left|\frac{l_z(D^4F(0)(z^4))}{4!\|z\|^4}\right|\nonumber\\&\le(1-\alpha)\left(|h_1|+\frac{1}{6}\left((1-\alpha)^2|h_1|^3+3(1-\alpha)|h_1||h_2|+2|h_3|\right)\right).
\end{align}

In view of Lemma \ref{Lem1} we get
\begin{align}\label{eqn-2.39}
\mathcal{T}_1\le\frac{2}{3}(1-\alpha)(9-7\alpha+2\alpha^2).	
\end{align}

{\bf{B.}} Using \eqref{Th-21.9}--\eqref{Th-21.11} to \eqref{m-2} we have
\begin{align}\label{eqn-2.40}
\mathcal{T}_2&=\frac{(1-\alpha)^2}{6}\left|6h_1^2-3h_2^2+2h_1h_3-3(1-\alpha)h_1^2h_2-2(1-\alpha)^2h_1^4\right|\nonumber\\&\le\frac{(1-\alpha)^2}{6}\left(6|h_1|^2+3|h_2|^2+2(1-\alpha)^2|h_1|^4+2|h_1|\left|h_3-\frac{3}{2}(1-\alpha)h_1h_2\right|\right).
\end{align}
	
Since $0\le\alpha\le\frac{2}{3}$. Using Lemma \ref{Lem1} to \eqref{eqn-2.40}, we obtain 
\begin{align}\label{eqn-2.41}
\mathcal{T}_2\le\frac{2}{3}(1-\alpha)^2\left(9+8(1-\alpha)^2+2(2-3\alpha)\right)=\frac{2}{3}(1-\alpha)^2\left(21-22\alpha+8\alpha^2\right).
\end{align}

Therefore, using \eqref{eqn-2.39} and \eqref{eqn-2.41} to \eqref{eqn-2.37} we obtain the desired bound. To established sharpness we consider the function $F(z)$ defined by \eqref{Sh-1}. Figure~\ref{fig:T3-factors} illustrates how the bound $B_2(\alpha)$ of Theorem~\ref{Th-3} arises as a product of the two factor bounds $\mathcal T_1(\alpha)$ and $\mathcal T_2(\alpha)$ established in \eqref{eqn-2.39} and \eqref{eqn-2.41}.
\end{proof}	
\begin{figure}[htbp]
	\centering
	\begin{tikzpicture}
		\begin{axis}[
			width=0.82\textwidth,
			height=0.46\textwidth,
			xmin=0, xmax=0.6667,
			ymin=0, ymax=15,
			xlabel={$\alpha$},
			ylabel={Factor value},
			axis lines=left,
			grid=major,
			grid style={dashed, opacity=0.35},
			samples=200,
			domain=0:0.6667,
			legend style={
				at={(0.98,0.97)},
				anchor=north east,
				draw=black,
				fill=white,
				rounded corners=1pt,
				font=\small
			}
			]
			\addplot[very thick, blue]
			{ (2/3)*(1-x)*(9-7*x+2*x^2) };
			\addlegendentry{$\mathcal T_1(\alpha)=\frac23(1-\alpha)(9-7\alpha+2\alpha^2)$}
			
			\addplot[very thick, red]
			{ (2/3)*(1-x)^2*(21-22*x+8*x^2) };
			\addlegendentry{$\mathcal T_2(\alpha)=\frac23(1-\alpha)^2(21-22\alpha+8\alpha^2)$}
			
			\addplot[very thick, dashed, black, domain=0:0.6667]
			{ (2/3)*(1-x)*(9-7*x+2*x^2) * (2/3)*(1-x)^2*(21-22*x+8*x^2) };
			\addlegendentry{$B_2(\alpha)=\mathcal T_1(\alpha)\mathcal T_2(\alpha)$}
			
			\addplot[only marks, mark=*, mark size=2.2pt] coordinates {(0,6) (0,14)};
		\end{axis}
	\end{tikzpicture}
	\caption{Decomposition of the sharp bound $B_2(\alpha)$ of Theorem~\ref{Th-3}
		into its two factors $\mathcal T_1(\alpha)$ and $\mathcal T_2(\alpha)$
		(cf.\ \eqref{eqn-2.39}, \eqref{eqn-2.41}), valid for $0\le\alpha\le\frac23$.
		At $\alpha=0$, $\mathcal T_1(0)=6$ and $\mathcal T_2(0)=14$, whose product
		recovers the classical constant $84$ of Ali et al.\ \cite{Ali et al.}.}
	\label{fig:T3-factors}
\end{figure}
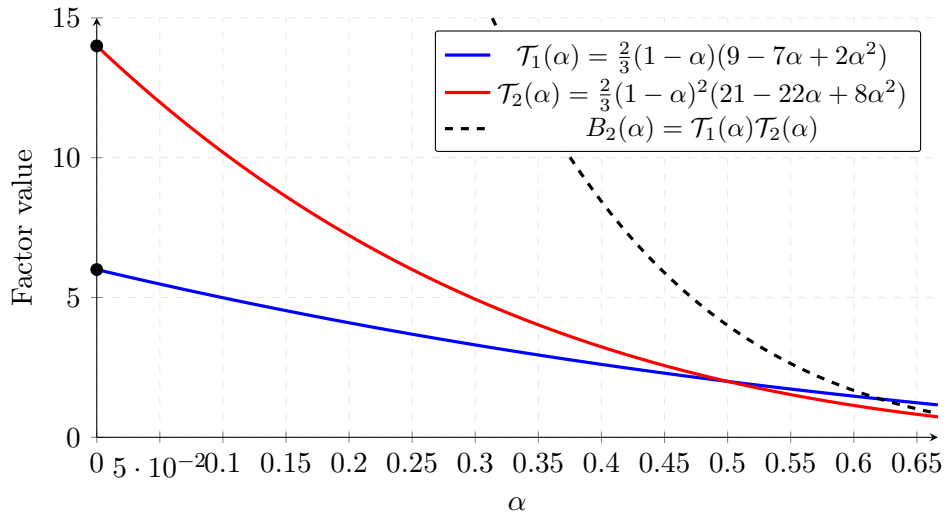
\begin{proof}[\bf Proof of Theorem \ref{Th-4}]
	The reduction used in the proof of Theorem~\ref{Th-2}
	applies here without change. In particular, the coefficient identities 	in \eqref{Th-22.8} remain valid. Using \eqref{Th-2x} and	\eqref{Th-22.11}, we obtain
	\eqref{Th-22.8}. Therefore, we have 
	
	\begin{align}\label{T-42.12}
		\displaystyle
		&\bigg|\bigg(\frac{1}{2!}D^2F_k(0)\left(z_0,\frac{D^2F(0)(z_0^2)}{2!}\right)\frac{||z||_{\infty}}{z_k}-\frac{1}{4!}D^4F_k(0)\left(z_0^3,\frac{D^2F(0)(z_0^2)}{2!}\right)\frac{||z||_{\infty}}{z_k}\nonumber\\&+\frac{2}{3!}D^3F_k(0)\left(z_0^2,\frac{D^3F(0)(z_0^3)}{3!}\right)\frac{||z||_{\infty}}{z_k}\bigg)\left(\frac{1}{2!}D^2F_k(0)(z_0,z_0)\frac{||z||_{\infty}}{z_k}-\frac{1}{4!}D^4F_k(0)(z_0,z_0^3)\frac{||z||_{\infty}}{z_k}\right)\bigg|\nonumber\\=&
		\bigg|\bigg(\frac{1}{2!}D^2F_k(0)\left(z_0,{Df(0)(z_0)z_0}\right)\frac{||z||_{\infty}}{z_k}-\frac{1}{4!}D^4F_k(0)\left(z_0^3,{Df(0)(z_0)z_0}\right)\frac{||z||_{\infty}}{z_k}\nonumber\\&+\frac{2}{3!}D^3F_k(0)\left(z_0^2,\frac{D^2f(0)(z_0^2)z_0}{2!}\right)\frac{||z||_{\infty}}{z_k}\bigg)\left(\frac{1}{2!}D^2F_k(0)(z_0^2)\frac{||z||_{\infty}}{z_k}-\frac{1}{4!}D^4F_k(0)(z_0^4)\frac{||z||_{\infty}}{z_k}\right)\bigg|\nonumber\\=&
		\bigg|\bigg({Df(0)(z_0)}\frac{1}{2!}D^2F_k(0)\left(z_0^2\right)\frac{||z||_{\infty}}{z_k}-{Df(0)(z_0)}\frac{1}{4!}D^4F_k(0)\left(z_0^4\right)\frac{||z||_{\infty}}{z_k}\nonumber\\&+2\left(\frac{D^2f(0)(z_0^2)}{2!}\right)\frac{1}{3!}D^3F_k(0)\left(z_0^3\right)\frac{||z||_{\infty}}{z_k}\bigg)\left(Df(0)(z_0)-\frac{1}{3!}D^3f(0)(z_0^3)\right)\bigg|\nonumber\\=&
			\bigg|\bigg({\left(Df(0)(z_0)\right)^2}-{Df(0)(z_0)}\frac{D^3f(0)\left(z_0^3\right)}{3!}+2\left(\frac{D^2f(0)(z_0^2)}{2!}\right)^2\bigg)\left(Df(0)(z_0)-\frac{1}{3!}D^3f(0)(z_0^3)\right)\bigg|\nonumber\\&=\hat{\mathcal{T}_1}\hat{\mathcal{T}_2},
	\end{align}
	where,
	\begin{align}
		\displaystyle\label{m-3}
	&\hat{\mathcal{T}_1}=\left|Df(0)(z_0)-\frac{1}{3!}D^3f(0)(z_0^3)\right|\\
	\displaystyle\label{m-4}
	&\hat{\mathcal{T}_2}=\left|{\left(Df(0)(z_0)\right)^2}-{Df(0)(z_0)}\frac{D^3f(0)\left(z_0^3\right)}{3!}+2\left(\frac{D^2f(0)(z_0^2)}{2!}\right)^2\right|.
	\end{align}

   {\bf{A.}} Using \eqref{Th-22.8} to \eqref{m-3}, we get
	\begin{align}\label{m-5}
		\displaystyle
		&\hat{\mathcal{T}_1}\le\left|Df(0)(z_0)\right|+\left|\frac{1}{3!}D^3f(0)(z_0^3)\right|\nonumber\\&\le(1-\alpha)\left(|h_{k1}|+\frac{1}{6}\left((1-\alpha)^2|h_{k1}|^3+3(1-\alpha)|h_{k1}||h_{k2}|+2|h_{k3}|\right)\right).
	\end{align}
	Thus, by Lemma \ref{Lem1} we get
	\begin{align*}
		\hat{\mathcal{T}_1}\le\frac{2}{3}(1-\alpha)(9-7\alpha+2\alpha^2).
	\end{align*}
	Therefore, by a similar proof as we did in Theorem \ref{Th-2},	we obtain the desired bound.
	\begin{align}\label{m-7}
		\displaystyle
		\left\|\frac{1}{2!}D^2F(0)(z,z)-\frac{1}{4!}D^4F(0)(z,z^3)\right\|_{\infty}\leq
		\frac{2}{3}(1-\alpha)(9-7\alpha+2\alpha^2)||z||_{\infty}.
	\end{align}
	{\bf{B.}} Using \eqref{Th-22.8} and \eqref{m-4}, we obtain
	\begin{align}\label{m-6}
		&\hat{\mathcal{T}_2}=\bigg|{\left(Df(0)(z_0)\right)^2}-{Df(0)(z_0)}\frac{D^3f(0)\left(z_0^3\right)}{3!}+2\left(\frac{D^2f(0)(z_0^2)}{2!}\right)^2\bigg|\nonumber\\&=\frac{(1-\alpha)^2}{6}\left|6h_{k1}^2-3h_{k2}^2+2h_{k1}h_{k3}-3(1-\alpha)h_{k1}^2h_{k2}-2(1-\alpha)^2h_{k1}^4\right|\nonumber\\&\le\frac{(1-\alpha)^2}{6}\left(6|h_{k1}|^2+3|h_{k2}|^2+2(1-\alpha)^2|h_{k1}|^4+2|h_{k1}|\left|h_{k3}-\frac{3}{2}(1-\alpha)h_{k1}h_{k2}\right|\right).
	\end{align}
	Since $0\le\alpha\le\frac{2}{3}$. Using the Lemma \ref{Lem1} to \eqref{m-6}, we have
	\begin{align*}
		\hat{\mathcal{T}_2}\le\frac{2}{3}(1-\alpha)^2(21-22\alpha+8\alpha^2).
	\end{align*}
	Hence, similar argument used in the proof of Theorem \ref{Th-2}, we obtain the desired bound
	\begin{align}\label{m-8}
		\displaystyle
		&\bigg\|\frac{1}{2!}D^2F(0)\left(z,\frac{D^2F(0)(z^2)}{2!}\right)-\frac{1}{4!}D^4F(0)\left(z^3,\frac{D^2F(0)(z^2)}{2!}\right)\nonumber\\&+\frac{2}{3!}D^3F(0)\left(z^2,\frac{D^3F(0)(z^3)}{3!}\right)\bigg\|_{\infty}\leq
		\frac{2}{3}(1-\alpha)^2(21-22\alpha+8\alpha^2)||z||_{\infty}.
	\end{align}
	Therefore, from \eqref{m-7} and \eqref{m-8}, we get the desired bound. To prove the sharpness we consider the function $F(z)$ defined by \eqref{Sh-2}.
\end{proof}

\vspace{5mm}

\noindent\textbf{Conflict of interest:} The authors declare that there is no conflict  of interest regarding the publication of this paper.\vspace{1.2mm}

\noindent {\bf Funding:} Not Applicable.\vspace{1.2mm}

\noindent\textbf{Data availability statement:}  Data sharing not applicable to this article as no datasets were generated or analysed during the current study.\vspace{1.2mm}

\noindent {\bf Authors' contributions:} All the authors have equal contributions in preparation of the manuscript.

\end{document}